\documentclass[11pt]{amsart}
\usepackage{amssymb,amsmath,graphicx}
\usepackage{tikz}\usetikzlibrary{matrix,arrows}

\newtheorem{thm}{Theorem}[section]
\newtheorem{mainthm}{Theorem}

\newtheorem{lem}[thm]{Lemma}

\newtheorem{cor}[thm]{Corollary}

\theoremstyle{definition}
\newtheorem{defn}[thm]{Definition}
\newtheorem{rem}[thm]{Remark}

\newtheorem*{ack}{Acknowledgments}

\newcommand{\F}{\mathbb{F}}
\newcommand{\Z}{\mathbb{Z}}
\newcommand{\R}{\mathbb{R}}
\newcommand{\C}{\mathbb{C}}
\newcommand{\sph}{\mathbb{S}}

\newcommand{\aut}{\textsc{Aut}}
\newcommand{\cay}{\textsc{Cay}}
\newcommand{\pure}{P}
\newcommand{\good}{$[\chi] \in \Sigma^1(\pure_n)$}

\newcommand{\onto}{\twoheadrightarrow}

\begin{document}

\title[$\Sigma^1(\pure_n)$]{The BNS-invariant for the pure braid groups}

\author[Koban]{Nic~Koban}
\address{Dept.~of Math., University of Maine Farmington, Farmington ME 04938}
\email{nicholas.koban@maine.edu}

\author[McCammond]{Jon~McCammond}
\address{Dept.~of Math., UC Santa Barbara, Santa Barbara CA 93106}
\email{jon.mccammond@math.ucsb.edu}

\author[Meier]{John~Meier}
\address{Dept.~of Math., Lafayette College, Easton PA 18042}
\email{meierj@lafayette.edu}

\date{\today}

\begin{abstract}
  In 1987 Bieri, Neumann and Strebel introduced a geometric invariant
  for discrete groups.  In this article we compute and explicitly
  describe the BNS-invariant for the pure braid groups.
\end{abstract}

\subjclass[2010]{20F65}
\keywords{BNS invariant, sigma invariants, pure braid groups}
\maketitle

In 1987 Robert Bieri, Walter Neumann, and Ralph Strebel introduced a
geometric invariant of a discrete group that is now known as its BNS
invariant \cite{BNS87}.  For finitely generated groups the invariant
is a subset of a sphere associated to the group called its character
sphere.  They proved that their invariant is an open subset of the
character sphere and that it determines which subgroups containing the
commutator subgroup are finitely generated.  In particular, the commutator
subgroup is itself finitely generated if and only if the BNS invariant
is the entire character sphere.  For fundamental groups of smooth
compact manifolds, the BNS-invariant contains information about the
existence of circle fibrations of the manifold and for fundamental
groups of $3$-dimensional manifolds, the BNS-invariant can be
described in terms of the Thurston norm.  Given these connections, it
is perhaps not surprising that the BNS-invariant is typically somewhat
difficult to compute.  It has been completely described for some 
infinite families of groups, including: one-relator groups \cite{Br87},
right-angled Artin groups \cite{MeVW95}, and the pure symmetric
automorphism groups of free groups \cite{Or00}.  In this article we
combine aspects of the proofs of these earlier results to compute and
explicitly describe the BNS-invariant for the pure braid groups.

\begin{mainthm}\label{thm:main}
  The BNS-invariant for the pure braid group $\pure_n$ is the
  complement of a union of  $\pure_3$-circles and the
  $\pure_4$-circles in its character sphere. There are exactly
  $\binom{n}{3} + \binom{n}{4}$ such circles.
\end{mainthm}

The names ``$\pure_3$-circle'' and ``$\pure_4$-circle'' are introduced
here in order to make our main result easier to state.  Their
definitions are given in Section~\ref{sec:complement}.  

\vskip 10pt

Our computation of $\Sigma^1(\pure_n)$ has a striking connection to the previously 
computed resonance variety for the pure braid groups 
(see Proposition 6.9 in \cite{CoSu99}).
The \emph{resonance variety} of a group is computed from
the structure of its cohomology ring. (For background information 
and full definitions,
see \cite{Su11}.)  In fact, there are
many resonance varieties just as there are many $\Sigma$ invariants,
but we are concerned here with the simplest forms of each.  In general
there is only a weak connection between the resonance variety of a
group and its BNS-invariant, but when certain conditions are met, the
resonance variety is contained in the complement of the BNS-invariant
\cite{PaSu10}.  In some interesting cases it is known that the complement
of the first resonance variety is equal to the first BNS invariant (see \cite{PaSu06} and \cite{Co09}).
In section 9.9 of \cite{Su11} it was asked if this equality holds for 
 fundamental groups of complements of hyperplane arrangements in $\C^n$.  The pure 
braid groups are perhaps the best
known example of an arrangement group, and so our result shows that 
this equality does hold in this case.
An example 
presented in \cite{Su12}---constructed by deleting one hyperplane
 from a reflection arrangement---demonstrates that  this equality does not always hold
for arrangement groups.  

\vskip 10pt

The article is structured as follows.  The first three sections
contain basic results about BNS-invariants, pure braid groups, and
graphs.  The fourth section finds several circles of characters in the
complement of the invariant for the pure braid groups.  The fifth
section establishes a series of reduction lemmas which collectively
show that every other character is contained in the invariant, thereby
completing the proof.

\begin{ack}
  The authors thank Ralph Strebel for requesting a
  description of the BNS-invariant for the pure braid groups some time
  ago; his continuing encouragement has helped bring this work to
  completion.  We  also  thank Alex Suciu for pointing out
  the relationship between our main result and the resonance variety
  of the pure braid groups.
\end{ack}

\section{BNS invariants}\label{sec:bns}

In this section we recall the definition of the BNS-invariant and
discuss two standard techniques used to compute them.

\begin{defn}[BNS-invariant]\label{def:bns}
  Let $G$ be a finitely generated group. A \emph{character of $G$} is
  a group homomorphism from $G$ to the additive reals and the set of
  all characters of $G$ is an $n$-dimensional real vector space where
  $n$ is the $\mathbb{Z}$-rank of the abelianization of $G$.  Let $I
  \subset G$ be a generating set and let $\cay(G,I)$ denote the right
  Cayley graph of $G$ with respect to $I$.  For any character $\chi$
  we let $\cay_\chi(G,I)$ denote the full subgraph of $\cay(G,I)$
  determined by the vertices whose $\chi$-values are non-negative.  The
  property that the BNS-invariant captures is whether or not
  $\cay_\chi(G,I)$ is connected.  It is somewhat surprising, but
  nonetheless true, that whether or not $\cay_\chi(G,I)$ is connected
  is independent of the choice of finite generating set $I$ and thus
  only depends on $\chi$.  It is much easier to see that this property
  is preserved when $\chi$ is composed with a dilation of $\R$.  
  As a consequence, one can replace characters with
  equivalence classes of characters where equivalence is defined by
  composition with dilations by positive real numbers $r$.  The set of
  equivalence classes is identified with the unit sphere in $\R^n$ and
  called the \emph{character sphere of $G$}: 
  \[
S(G) = \{ \chi~|~\chi \in {\rm Hom}(G,\mathbb{R}) - \{0\} \}/\chi \sim r\cdot \chi
\]
where $r \in (0, \infty) \subset \R$.  The
  \emph{Bieri-Neumann-Strebel-invariant of $G$} is the set of
  equivalence classes of characters $[\chi]$ such that
  $\cay_\chi(G,I)$ is connected.  We write $\Sigma^1(G)$ for this
  invariant and we write $\Sigma^1(G)^c$ for the complementary portion
  of the character sphere.  
\end{defn}

\begin{rem}
The superscript ``$1$" in the notation $\Sigma^1(G)$ indicates
  that there are generalizations of these definitions.  The first of these was introduced 
by Bieri and Renz in \cite{BiRe88}.  These invariants have also been described
in terms of Novikov homology \cite{Bi07}, and so our result
relates to the work in  \cite{KoPa12}.
Bieri and Geoghegan have presented extensions of the original
definition that are applicable to group
actions on non-positively curved spaces \cite{BiGe04}.  
\end{rem}

We use a common algebra metaphor to describe the images of elements
under $\chi$.  We say that $g$ \emph{lives} or \emph{survives} if
$\chi(g)$ is not zero and that $g$ \emph{dies} or \emph{is killed}
when $\chi(g)$ is zero.  There are two main techniques that we use to
compute BNS-invariants.  One to show that characters are in the
complement and the other to show that characters are in the invariant.

\begin{lem}[Epimorphisms]\label{lem:epis}
  Let $\phi:G \onto H$ be an epimorphism between finitely generated
  groups.  If $\psi$ is a character of $H$ and $\chi$ is the character
  of $G$ defined by $\chi = \psi \circ \phi$, then $[\chi] \in
  \Sigma^1(G)$ implies $[\psi] \in \Sigma^1(H)$ and $[\psi] \in
  \Sigma^1(H)^c$ implies $[\chi] \in \Sigma^1(G)^c$.
\end{lem}

\begin{proof}
  If we choose generating sets $I$ and $J$ for $G$ and $H$
  respectively so that $\phi(I) = J$, then the epimorphism $\phi$
  naturally extends to a continuous map from the Cayley graph of $G$
  onto the Cayley graph of $H$ which then restricts to a continuous
  map from $\cay_\chi(G,I)$ onto $\cay_\psi(H,J)$.  Since the
  continuous image of a connected space is connected, $\cay_\chi(G,I)$
  connected implies $\cay_\psi(H,J)$ is connected and $\cay_\psi(H,J)$
  disconnected implies $\cay_\chi(G,I)$ is disconnected.
\end{proof}

Lemma~\ref{lem:epis} is primarily used is to find characters in
$\Sigma^1(G)^c$.  For each homomorphism $\phi$ from $G$ onto a simpler
group $H$ whose BNS-invariant is already known, the preimage of
$\Sigma^1(H)^c$ under $\phi$ is a subset of $\Sigma^1(G)^c$.  A second
use of Lemma~\ref{lem:epis} is that it implies $\Sigma^1(G)$ is
invariant under automorphisms of $G$.  For any finitely generated
group $G$, precomposition defines a natural right action of $\aut(G)$
on the character sphere with $[\chi]\cdot \alpha$ defined to be $[\chi
  \circ \alpha]$ for all $\alpha \in \aut(G)$ and all characters
$\chi$.  For each automorphism $\alpha \in \aut(G)$,
Lemma~\ref{lem:epis} can be applied twice, once with $\phi=\alpha$ and
a second time with $\phi=\alpha^{-1}$ to obtain the following
immediate corollary.

\begin{cor}[Automorphisms]\label{cor:auts}
  For any finitely generated group $G$, the subsets $\Sigma^1(G)$ and
  $\Sigma^1(G)^c$ are invariant under the natural right action of
  $\aut(G)$ on the character sphere of $G$.
\end{cor}

There is an alternative description of $\Sigma^1(G)$ using $G$-actions on
$\R$-trees.

\begin{defn}[Actions on $\R$-trees]\label{def:r-trees}
 Suppose $G$ acts by isometries on an $\mathbb{R}$-tree $T$ and let
 $\ell:G \to \mathbb{R}^+$ be the corresponding length function.  It
 is \emph{non-trivial} if there are no global fixed points. It is
 \emph{exceptional} if there are no invariant lines.  It is
 \emph{abelian} if there exists a character $\chi$ of $G$ such that
 the translation length function $\ell(g)$ equals the absolute value
 of $\chi(g)$ for all $g\in G$.  When this occurs we say that this
 action is \emph{associated to $\chi$}.
\end{defn}

The following lemma describes $\Sigma^1(G)$ in these terms.

\begin{lem}[Actions and characters]\label{lem:brown}
  Let $\chi$ be a character of a group $G$.  There exists an
  exceptional non-trivial abelian $G$-action on an $\R$-tree
  associated to $\chi$ if and only if $[\chi] \in \Sigma^1(G)^c$.
\end{lem}

A proof of Lemma~\ref{lem:brown} can be found in \cite{Br87}.  For
each $g \in G$, let $T_g$ denote the characteristic subtree of $g$.
When $g$ is elliptic, $T_g$ is its fixed point set, and when $g$ is
hyperbolic, $T_g$ is the axis of $g$.  There are two main facts about
characteristic subtrees that we need: (1) if $g$ and $h$ are commuting
hyperbolic isometries then $T_g = T_h$ and (2) if $g$ commutes with a
hyperbolic isometry $h$ then $T_g \supset T_h$.  Both properties are
discussed in \cite{Or00}.

\begin{defn}[Commutation]\label{def:commute}
  For any subset $J$ of a group $G$ there is a natural graph that
  records which elements commute.  It has a vertex set indexed by $J$
  and two distinct vertices are connected by an edge if only if the
  corresponding elements of $J \subset G$ commute.  We call this the
  \emph{commuting graph of $J$ in $G$} and denote it by $C(J)$.
\end{defn}

\begin{defn}[Domination]\label{def:dominate}
  Let $I$ and $J$ be subsets of a group $G$.  We say that \emph{$J$
    dominates $I$} if every element of $I$ commutes with some element
  of $J$.  Since elements commute with themselves, this is equivalent
  to the assertion that every element in $I\setminus J$ commutes with
  some element of $J$.
\end{defn}

\begin{lem}[Connected and Dominating]\label{lem:condom}
  Let $\chi$ be a character of a group $G$.  If there exist subsets
  $I$ and $J$ in $G$ such that $J$ survives under $\chi$, $C(J)$ is
  connected, $J$ dominates $I$, and $I$ generates $G$, then $[\chi]
  \in \Sigma^1(G)$.
\end{lem}

\begin{proof}
  Suppose there is an abelian action of $G$ on an $\R$-tree $T$
  associated to $\chi$.  Since elements of $J$ survive under $\chi$,
  each is realized as a hyperbolic isometry of the tree.  Because  $C(J)$
 is connected, all of these isometries share a common
  characteristic subtree $T' = T_j$ for all $j\in J$.  Because $J$
  dominates $I$, each element $i \in I$ commutes with a hyperbolic
  isometry $j \in J$ which implies $T_i \supset T_j = T'$ for all $i\in
  I$.  Finally, $I$ generates $G$, so the line $T'$ is invariant under
  all of $G$, the action is not exceptional, and $[\chi] \in
  \Sigma^1(G)$ by Lemma~\ref{lem:brown}.
\end{proof}

Lemma~\ref{lem:condom} is our primary tool for finding characters in
$\Sigma^1(G)$.  To illustrate its utility we include an application:
characters in the complement of the BNS-invariant must kill the center of
the group.

\begin{cor}[Central elements]\label{cor:center}
  If $\chi$ is a character of a group $G$ and $\chi$ is not
  identically zero on the center of $G$ then $[\chi]$ is in
  $\Sigma^1(G)$.
\end{cor}

\begin{proof}
  Let $I$ be any generating set for $G$ and let $J = \{g\}$ where $g$
  is a central element that lives under $\chi$.  The graph $C(J)$ is
  connected because it only has one vertex and $J$ dominates $I$
  because $g$ is central.  Lemma~\ref{lem:condom} completes the proof.
\end{proof}

\section{Pure braid groups}\label{sec:pure}

Next we recall some basic properties of the pure braid groups.

\begin{defn}[Pure braid groups]\label{def:braid-arrangement}
  Let $\C^n$ be an $n$-dimensional complex vector space with a fixed
  basis and let $H_{ij}$ be the hyperplane in $\C^n$ defined by the
  equation $z_i=z_j$.  The set $\{H_{ij}\}$ of all such hyperplanes is
  called the \emph{braid arrangement} and it is one of the standard
  examples in the theory of hyperplane arrangements.  The fundamental
  group of the complement of the union of these hyperplanes is called
  the \emph{pure braid group $\pure_n$}:
\[
\pure_n = \pi_1\left(\C^n \setminus \{H_{ij}\}\right).
\]
\end{defn}

\begin{defn}[Points in the plane]
  There is a standard $2$-dimensional way to view points in the
  complement of the braid arrangement.  For each vector in $\C^n$ we
  have a configuration of $n$ labeled points in the complex plane.
  More concretely, the point $p_i$ in $\C$ is meant to indicate the
  value of the $i$-th coordinate of the vector and avoiding the
  hyperplanes $H_{ij}$ corresponds to configurations where these
  points are distinct.  Paths in the hyperplane complement correspond
  to motions of these $n$ labeled points in the plane which remain
  distinct throughout.  If we trace out these motions over time in a
  product of $\C$ with a time interval, then the points become strands
that braid.
\end{defn}

\begin{defn}[Basepoint]
  Computing the fundamental group of a hyperplane complement requires
  a choice of basepoint.  We select one corresponding to the
  configuration where the $n$ labeled points are equally spaced around
  the unit circle and $p_1$ through $p_n$ occur consecutively as one
  proceeds in a clockwise direction.  See Figure~\ref{fig:convex} for
  an illustration.  Loops representing elements of the fundamental
  group are motions of these points which start and end at this
  particular configuration.
\end{defn}

\begin{figure}
  \begin{tikzpicture}[scale=3]
    \definecolor{lightblue}{rgb}{.85,.85,1};
    \definecolor{mediumblue}{rgb}{.7,.7,1};
    \fill[color=lightblue]
    (0:1cm)--(40:1cm)--(80:1cm)--(120:1cm)--(160:1cm)--
    (200:1cm)--(240:1cm)--(280:1cm)--(320:1cm)--cycle;
    \fill[color=mediumblue]  (80:1cm)--(320:1cm)--(280:1cm)--cycle;
    \draw[-] (80:1cm)--(320:1cm)--(280:1cm)--cycle;
    \draw[-] (40:1cm)--(0:1cm);
    \draw[-] (120:1cm)--(200:1cm);
    \draw[-] (160:1cm)--(240:1cm);
    \draw (0:1cm) node [anchor=west] {$p_3$};
    \draw (40:1cm) node [anchor=south west] {$p_2$};
    \draw (80:1cm) node [anchor=south] {$p_1$};
    \draw (120:1cm) node [anchor=south east] {$p_9$};
    \draw (160:1cm) node [anchor=east] {$p_8$};
    \draw (200:1cm) node [anchor=east] {$p_7$};
    \draw (240:1cm) node [anchor=north east] {$p_6$};
    \draw (280:1cm) node [anchor=north] {$p_5$};
    \draw (320:1cm) node [anchor=north west] {$p_4$};
    \foreach \x in {0,1,2,3,4,5,6,7,8} {\fill (40*\x:1cm) circle (.3mm);}
    \foreach \x in {0,1,2,3,4,5,6,7,8} {\fill[color=white] (40*\x:1cm) circle
      (.2mm);}
  \end{tikzpicture}
  \caption{Nine points in convex position.\label{fig:convex}}
\end{figure}
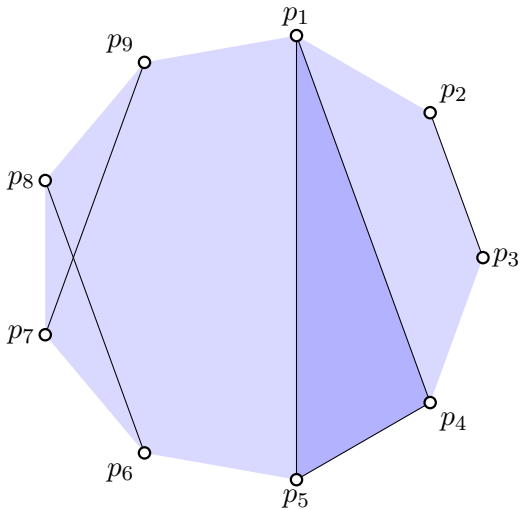

\begin{defn}[Swing generators]\label{def:swings}
  For each set $A \subset \{1,2,\ldots,n\}$ of size at least $2$ there
  is an element of $\pure_n$ obtained as follows.  Move the
  points corresponding to the elements of $A$ directly towards the
  center of their convex hull.  Once they are near to each other,
  rotate the small disk containing them one full twist in a clockwise
  direction and then return these points to their original position
  traveling back the way they came.  When $A$ is small we write
  $S_{ij}$ or $S_{ijk}$ with the subscripts indicating the points
  involved.  In \cite{MaMc09} these elements are called \emph{swing
    generators}.  One key property of the swing generator $S_A$ is
  that it can be rewritten as a product of the swing generators $S_{ij}$
  with $\{i,j\} \subset A$.  The order in which they are multiplied is
  important, but it rarely arises in this context.  As an illustration
  of this type of factorization, the element $S_{123}$ is equal to the
  product $S_{12} S_{13} S_{23}$ and to $S_{13} S_{23} S_{12}$ and to
  $S_{23} S_{12} S_{13}$.  (For the record we are composing these
  elements left-to-right as is standard in the study of braid groups.)
  This means that the $\binom{n}{2}$ swing generators which only
  involve two points are sufficient to generate $\pure_n$ and we call
  this set the \emph{standard generating set for this arrangement}.
\end{defn}

A presentation for the pure braid group was given by Artin in
\cite{Art47} and more recent geometric variations are given by
Margalit and McCammond in \cite{MaMc09}.  For our purposes the most
relevant fact about these various presentations is that all of their
relations become trivial when abelianized.  This immediately implies
the following:

\begin{lem}[Pure braid characters]\label{lem:pn-characters}
  The abelianization of the pure braid group is free abelian with the
  images of the standard generators as a basis.  As a consequence,
  there are no restrictions on the tuples of values a character may
  assign to the standard generators.   Thus $\mbox{Hom}(\pure_n, \R)
\simeq \R^{\binom{n}{2}}$ and the character sphere
has dimension $\binom{n}{2} - 1$:
\[
S(\pure_n) = \sph^{\binom{n}{2} - 1}.
\]
\end{lem}

There are two aspects of the pure braid group that are particularly
useful in this context.  The first is that many pairs of swing
generators commute and the second is that there is an automorphism of
$\pure_n$ whose net effect is to permute the labeled points in the
plane without changing the character values on the corresponding
standard generators.

\begin{rem}[Commuting swings]
  Let $S_A$ and $S_B$ be two swing generators in $\pure_n$.  The
  elements $S_A$ and $S_B$ commute when $A \subset B$, $B \subset A$,
  or the convex hull of the points in $A$ does not intersect the
  convex hull of the points in $B$ \cite{MaMc09}.  For example
  $S_{23}$ and $S_{145}$ commute as do $S_{14}$ and $S_{145}$, but
  $S_{68}$ and $S_{79}$ do not.  See Figure~\ref{fig:convex}. One
  consequence of this property is that the element $\Delta = S_A$ with
  $A = \{1,2,\ldots,n\}$ is central in $\pure_n$.  In fact $\Delta$
  generates the center.
\end{rem}

  There is an obvious action of the symmetric group on the braid
  arrangement which permutes coordinates.  And since the union of the
  hyperplanes $H_{ij}$ contains all the points fixed under the action
  of a nontrivial permutation, the action on the complement is
  free. If we quotient by this action, the effect is to remove the
  labels from the points in the plane and the fundamental group of the
  quotient is the braid group.  This relationship is captured by the
  fact that there is a natural epimorphism from the braid group to the
  symmetric group (where the image of a braid is the way it permutes
  its strands) and its kernel is the pure braid group.

  The symmetric group action on the braid arrangement essentially
  changes the basepoint in the hyperplane complement and permutes the
  labels on the points in the plane.  For each such basepoint there is
  a set of swing generators but recall that there is no natural
  isomorphism between the fundamental group of a connected space at
  one basepoint and its fundamental group at another.  To create an
  isomorphsim one selects a path from the one to the other and then
  conjugates by this path.  In our case such a path projects to a loop
  in the quotient by the symmetric group action and thus represents an
  element of the braid group.  In particular, the resulting
  isomorphism between the fundamental groups is induced by an inner
  automorphism of the braid group which descends to an automorphism of
  its pure braid subgroup.  By Corollary~\ref{cor:auts} this
  automorphism of $\pure_n$ does not alter the BNS-invariant or its
  complement.

  It does, however, change the standard generating set.  If we keep
  track of the motion of the points in the plane dictated by the path
  between the basepoints, we find that the straight line segment
  between $p_i$ and $p_j$ used to define $S_{ij}$ becomes an embedded
  arc between the images of these points that is typically very
  convoluted.  In other words, the image of the original swing
  generator $S_{ij}$ is a nonstandard generator where the points $p_i$
  and $p_j$ travel along the twisted embedded arc from either end
  until they are very close, they then rotate fully around each other
  clockwise and then they return the way they came.  Despite the fact
  that the image of a standard generator is no longer standard, it is
  true that the new nonstandard generator is conjugate in $\pure_n$ to
  the standard generator between these two points.  In particular, for
  any character $\chi$, the $\chi$-value of a standard generator
  $S_{ij}$ is equal to the $\chi$-value of the standard generator
  between the images of $p_i$ and $p_j$ under this automorphism of
  $\pure_n$.

We conclude this section with a discussion of epimorphisms between
pure braid groups.

\begin{defn}[Natural projections]\label{def:phi_A}
  For every subset $A \subset \{1,2,\ldots,n\}$ of size $k$ there is a
  natural projecting epimorphism $\phi_A:\pure_n \onto \pure_k$ which
  can be described topologically as ``forgetting'' what happens to the
  points not in $A$.  Algebraically $\phi_A$ sends a standard
  generator $S_{ij}$ to zero unless both endpoints belong to $A$.
  This produces $\binom{n}{k}$ epimorphisms from $\pure_n$ onto
  $\pure_k$ which are all distinct.  The situations with $k=3$ or
  $k=4$ are particularly important here and we denote these maps by
  $\phi_{ijk}$ and $\phi_{ijkl}$ where the subscripts indicate the
  points contained in $A$.
\end{defn}

The fact that the complete graph on $4$ vertices is planar leads to a
nice presentation for $\pure_4$ and a surprising projection from $\pure_4$
onto $\pure_3$.

\begin{figure}
  \begin{tikzpicture}[scale=3]
    \draw[-] (330:1cm)--(210:1cm)--(0,0)--(90:1cm)--cycle;
    \draw[-] (210:1cm)--(90:1cm);
    \draw[-] (330:1cm)--(0,0);
    \draw (0,0) node [below=2pt] {$p_3$};
    \draw (90:1cm) node [above=2pt] {$p_2$};
    \draw (210:1cm) node [anchor=north east] {$p_1$};
    \draw (330:1cm) node [anchor=north west] {$p_4$};
    \fill[color=white] (150:5mm) circle (.6mm);\draw (150:5mm) node {$a$};
    \fill[color=white] (210:5mm) circle (.6mm);\draw (210:5mm) node {$b$};
    \fill[color=white] (90:5mm) circle (.6mm);\draw (90:5mm) node {$c$};
    \fill[color=white] (330:5mm) circle (.6mm);\draw (330:5mm) node {$d$};
    \fill[color=white] (30:5mm) circle (.6mm);\draw (30:5mm) node {$e$};
    \fill[color=white] (270:5mm) circle (.6mm);\draw (270:5mm) node {$f$};
    \fill (0,0) circle (.3mm);\fill[color=white] (0,0) circle (.2mm);
    \fill (90:1cm) circle (.3mm); \fill[color=white] (90:1cm) circle (.2mm);
    \fill (210:1cm) circle (.3mm); \fill[color=white] (210:1cm) circle (.2mm);
    \fill (330:1cm) circle (.3mm); \fill[color=white] (330:1cm) circle (.2mm);
  \end{tikzpicture}
  \caption{A labeled planar embedding of $K_4$.\label{fig:k4}}
\end{figure}
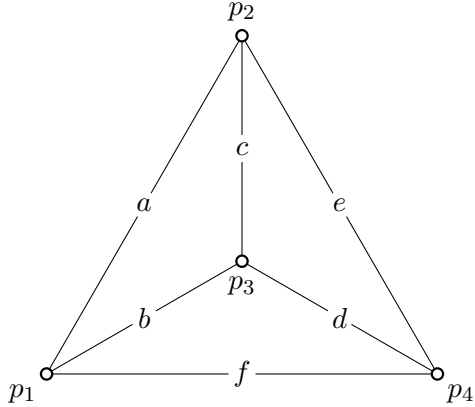

\begin{defn}[Planar presentation of $\pure_4$]\label{def:p4-pres}
  If we pick a basepoint for the braid arrangement that corresponds to
  the configuration of points shown in Figure~\ref{fig:k4}, then the
  six straight segments connecting them pairwise produce six swing
  generators that form a nonstandard generating set for $\pure_4$.  We
  denote these $a$ through $f$ as indicated.  The following is a
  presentation for $\pure_4$ in this generating set:
  \begin{equation*} 
    \pure_4 \cong \left< a,b,c,d,e,f\ \begin{array}{|l}
    abc=bca=cab,\ ad=da\\ cde=dec=ecd,\ be=eb\\ bfd=fdb=dbf,\ cf=fc\end{array}
    \right>
  \end{equation*}
  \noindent
  We call this the \emph{planar presentation of $\pure_4$}.  This
  presentation appears to be folklore: it is well-known to experts in
  the field but we cannot find a reference to it in the literature.
  Since it is straightforward to produce this presentation from one of
  the standard presentations, we omit the derivation.
\end{defn}

\begin{lem}[An unusual map]\label{lem:rho}
  There is a morphism $\rho:\pure_4 \onto \pure_3$ which sends the
  pairs of generators representing disjoint edges in the planar
  presentation of $\pure_4$ to the same standard generator of
  $\pure_3$.  Concretely, the function that sends both $a$ and $d$ to
  $a = S_{12}$, both $b$ and $e$ to $b= S_{13}$ and both $c$ and $f$
  to $c = S_{23}$ extends to such an epimorphism.
\end{lem}

\begin{proof}
  Since the image of this function is a generating set, the only thing
  to check is that images of the planar generators satisfy the planar
  relations.  This is clear since $a^2=a^2$, $b^2=b^2$, $c^2=c^2$
  and $abc=bca=cab$ in $\pure_3$.
\end{proof}

\section{Graphs}\label{sec:graphs}

In this section we record a few miscellaneous remarks related to
graphs that we use in the proof of the main result.  The first is the definition of an
auxilary graph that organizes information about a character, the
second is an elementary result from linear algebra, and the third is a
structural result about graphs that avoid a particular condition.

\begin{defn}[Graph of a character]\label{def:chi-graph}
  For each character $\chi$ of $\pure_n$ we construct a graph $K_\chi$
  that we call the \emph{graph of $\chi$}.  It is a subgraph of the
  complete graph $K_n$, it contains all vertices $v_i$ with $i$ in
  $\{1,2,\ldots, n\}$ and it contains the edge $e_{ij}$ from $v_i$ to
  $v_j$ if and only if the standard generator $S_{ij}$ survives under
  $\chi$.  When working with examples, it is convenient to add labels
  to the edges of $K_\chi$ which record the $\chi$-values of the
  corresponding standard generator.  For example, the labeled graph
  shown in Figure~\ref{fig:chi-graph} comes from a character which
  sends $S_{24}$ to $0$ and $S_{13}$ to $2$.  For any set $A \subset
  [n]$, the $\chi$-value of $S_A$ can be recovered from $K_\chi$ by
  adding up the labels on the edges with both endpoints in $A$.  Thus
  the character whose graph is shown in Figure~\ref{fig:chi-graph}
  sends $S_{124}$ to $-1$, $S_{123}$ to $0$ and $S_{1234}$ to $-3$.
\end{defn}

\begin{figure}
  \begin{tikzpicture}[scale=2.5]
    \draw[-] (1,0)--(0,1)--(0,0)--(1,0)--(1,1)--(0,1);
    \draw (0,1) node [anchor=south] {$v_1$};
    \draw (1,1) node [anchor=south] {$v_2$};
    \draw (1,0) node [anchor=north] {$v_3$};
    \draw (0,0) node [anchor=north] {$v_4$};
    \fill[color=white] (.5,0) circle (.6mm); \draw (.5,0) node {$1$};
    \fill[color=white] (.5,1) circle (.6mm); \draw (.5,1) node {$3$};
    \fill[color=white] (1,.5) circle (.6mm); \draw (1,.5) node {$-5$};
    \fill[color=white] (0,.5) circle (.6mm); \draw (0,.5) node {$-4$};
    \fill[color=white] (.5,.5) circle (.6mm); \draw (.5,.5) node {$2$};
    \foreach \x in {0,1} \foreach \y in {0,1} {\fill (\x,\y) circle (.3mm);}
  \end{tikzpicture}
  \caption{The graph of the character.\label{fig:chi-graph}}
\end{figure}
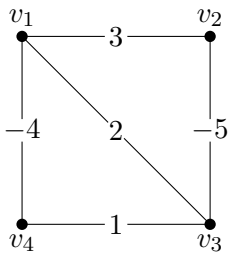

\begin{lem}[Triple sums]\label{lem:triple-sums}
  Let $\chi$ be a character for $\pure_4$.  If the four values
  $\chi(S_{123})$, $\chi(S_{124})$, $\chi(S_{134})$ and
  $\chi(S_{234})$ are all zero, then $\chi(S_{12}) = \chi(S_{34})$,
  $\chi(S_{13}) = \chi(S_{24})$, $\chi(S_{14}) = \chi(S_{23})$, and
  $\chi(S_{1234})=0$.
\end{lem}

\begin{proof}
  The proof is elementary linear algebra.  If we expand the four given
  values as sums over the edges of $K_\chi$ and then add all four
  equations together, we find that twice the sum over all six edges is
  zero.  If we then add two of the triangle equations and subtract the
  sum of six edges, we find that difference in the values of $\chi$ on
  a pair of disjoint edges is zero.  In other words, their values are
  equal.  This completes the proof.
\end{proof}

And finally, we consider a condition on a graph that arises in
Section~\ref{sec:invariant} and which has quite strong structural
implications.

\begin{lem}[Star or small]\label{lem:star-small}
  Let $\Gamma$ be a graph with no isolated vertices.  If $\Gamma$ does
  not contain an edge disjoint from two other edges, then all edges of
  $\Gamma$ have an endpoint in common, or they collectively have at
  most $4$ endpoints.  In other words, $\Gamma$ is a star or a
  subgraph of $K_4$.
\end{lem}

\begin{proof}
  If $\Gamma$ has a vertex $v$ of valence more than $3$, then edges
  ending at $v$ are its only edges.  Otherwise, the additional edge is
  disjoint from at least two of the edges with $v$ as an endpoint,
  contradicting our assumption.  If $\Gamma$ has a vertex of valence
  $3$, then these four vertices are the only vertices of $\Gamma$.
  Otherwise, there is an edge with only one endpoint in this set, it
  must end at one of the vertices other than $v$ (since the edges
  ending at $v$ are already accounted for) and thus it avoids two of
  the edges with $v$ as an endpoint, contradiction.  Finally, if
  $\Gamma$ only has vertices of valence $1$ and $2$, it is a
  collection of disjoint paths and cycles.  If there is a cycle, then
  there can be no other components and the cycle must have length at
  most $4$.  If there are multiple paths, there can only be two and
  they must both consists of a single edge.  If there is only one
  path, it can have at most $3$ edges.  This completes the proof
  since a cycle of length at most $4$, two paths of length $1$ and a
  single path of length $3$ are all subgraphs of $K_4$.
\end{proof}

\section{Characters in the Complement}\label{sec:complement}

In this section we recall the BNS-invariants for $\pure_2$ and
$\pure_3$ and then use the various epimorphisms between pure braid
groups to produce a series of circles in the complement of
$\Sigma^1(\pure_n)$ that we call $\pure_3$-circles and
$\pure_4$-circles.  Since $\pure_2$ is abelian, its invariant is
trivial to compute by Corollary~\ref{cor:center}.

\begin{lem}[$2$ points]\label{lem:p2}
  The group $\pure_2$ is isomorphic to $\Z$, its character sphere is
  $\sph^0$, the set $\Sigma^1(\pure_2)^c$ is empty and
  $\Sigma^1(\pure_2)$ includes both points.
\end{lem}

The group $\pure_3$ is only slightly more complicated.

\begin{lem}[$3$ points]\label{lem:p3}
  The group $\pure_3$ is isomorphic to $\F_2 \times \Z$, its character
  sphere is $\sph^2$, the set $\Sigma^1(\pure_3)^c$ is the equatorial
  circle defined by $\chi(\Delta) = 0$ and $\Sigma^1(\pure_3)$ is the
  complement of this circle.
\end{lem}

\begin{proof}
  The standard presentation for $\pure_3$ is $\langle a,b,c \mid
  abc=bca=cab \rangle$ where $a=S_{12}$, $b=S_{13}$ and $c=S_{23}$.
  If we add $d=\Delta=S_{123}$ as a generator and use the equation
  $abc=d$ to eliminate $c$ we obtain the following alternative
  presentation: $\pure_3 \cong \langle a,b,d \mid ad=da,\ bd=db
  \rangle$, from which it is clear that $\pure_3 \cong \F_2 \times \Z$
  with the free group $\F_2$ generated by $a$ and $b$ and the central
  $\Z$ generated by $d = \Delta$.  By Corollary~\ref{cor:center}
  characters in $\Sigma^1(\pure_3)^c$ must send $\Delta$ to zero.  On
  the other hand, those which do send $\Delta$ to zero are really
  characters of $\F_2$ and it is well-known that the BNS-invariant for
  a free group is empty.  Thus $[\chi] \in \Sigma^1(\pure_3)^c$ if and
  only if $\chi(\Delta) = 0$.
\end{proof}

It is the circle of characters that forms the complement of
$\Sigma^1(\pure_3)$ which produces multiple circles in the complement
of $\Sigma^1(\pure_n)$ for $n >3$.

\begin{defn}[$\pure_3$-circles and  $\pure_4$-circles]
  We say that $\chi$ is part of a \emph{$\pure_3$-circle} if there
  exists natural projection map $\phi_{ijk}$ (described in Definition~\ref{def:phi_A})
 and a character $\psi$ where $[\psi] \in \Sigma^1(\pure_3)^c$, such that 
$\chi = \psi \circ \phi_{ijk}$:
  \begin{equation*} \pure_n \stackrel{\phi_{ijk}}{\longrightarrow} \pure_3
  \stackrel{\psi}{\longrightarrow} \R .
  \end{equation*}
  \noindent
More concretely, $\chi$ is
  part of a $\pure_3$-circle if and only if all the endpoints of edges
  in $K_\chi$ belong to a three element subset $\{v_i,v_j,v_k\}$ and
  the value of $\chi(S_{ijk})$ is zero.  
  
In a similar fashion we say
  that $\chi$ is part of a \emph{$\pure_4$-circle} if there exists a
  triple of maps:
  \begin{equation*} 
    \pure_n \stackrel{\phi_{ijkl}}{\longrightarrow} \pure_4
    \stackrel{\rho}{\longrightarrow} \pure_3
  \stackrel{\psi}{\longrightarrow} \R 
  \end{equation*}
  \noindent
  whose composition is $\chi$ where $\phi_{ijkl}$ is one of the
  natural projection maps, 
  $\rho$ is the unusual map described in Lemma~\ref{lem:rho} and
  $[\psi] \in \Sigma^1(\pure_3)^c$.  More concretely,
  $\chi$ is part of a $\pure_4$-circle if and only if all the
  endpoints of edges in $K_\chi$ belong to a four element subset
  $\{v_i,v_j,v_k,v_l\}$, the equations $\chi(S_{ij}) = \chi(S_{kl})$,
  $\chi(S_{ik}) =\chi(S_{jl})$, $\chi(S_{il})=\chi(S_{jk})$ hold and
  the sum of these three shared values is zero.
\end{defn}

Using Lemma~\ref{lem:epis} we immediately conclude the following:

\begin{thm}[Characters in the complement]\label{thm:complement}
  Let $\chi$ be a character of $\pure_n$.  If $\chi$ is a part of a
  $\pure_3$-circle or a $\pure_4$-circle, then $[\chi] \in
  \Sigma^1(\pure_n)^c$.  This produces $\binom{n}{3} + \binom{n}{4}$
  circles in the complement.
\end{thm}

\begin{proof}
  The definitions of $\pure_3$-circles and $\pure_4$-circles ensure
  that Lemma~\ref{lem:epis} may be applied to $\chi$ to complete the
  proof.  The second assertion comes from the number of natural
  projections onto $3$ points plus the number of natural
  projections onto $4$ points.
\end{proof}

\section{Characters in the invariant}\label{sec:invariant}

In this final section we show that every character of $\pure_n$ that
is not part of a $\pure_3$-circle or a $\pure_4$-circle is in
$\Sigma^1(\pure_n)$.  We begin with a series of lemmas which follow,
directly or indirectly, from Lemma~\ref{lem:condom}.

\begin{lem}[Zero sum]\label{lem:zero-sum}
  If $\chi$ is a character of $\pure_n$ and $\chi(\Delta)$ is not
  zero, then \good.
\end{lem}

\begin{proof}
  Since $\Delta$ is central in $\pure_n$, this follows from
  Corollary~\ref{cor:center}.
\end{proof}

\begin{lem}[Disjoint triple]\label{lem:3disjont}
  If $\chi$ is a character of $\pure_n$ and $K_\chi$ contains three
  pairwise disjoint edges, then \good.
\end{lem}

\begin{proof}
  First permute the points so that $e_{12}$, $e_{34}$ and $e_{56}$
  are edges in $K_\chi$.  Then let $J = \{S_{12}, S_{34}, S_{56}\}$
  and let $I$ be the full standard generating set for this
  arrangement.  The graph $C(J)$ is a triangle and $J$ dominates $I$
  because every standard generator commutes with at least one element in $J$.
  Lemma~\ref{lem:condom} completes the proof.
\end{proof}

\begin{lem}[Disjoint from a pair]\label{lem:disjoint-pair}
  If $\chi$ is a character of $\pure_n$ and $K_\chi$ contains an edge
  disjoint from two other edges, then \good.
\end{lem}

\begin{proof}
  If all three edges are disjoint then Lemma~\ref{lem:3disjont}
  applies.  Otherwise, permute the points so that $e_{12}$,
  $e_{34}$ and $e_{45}$ are edges in $K_\chi$.  Then let $J =
  \{S_{12},S_{34},S_{45}\}$ and let $I$ be a modification of the
  standard generating set for this arrangement where $S_{14}$ and $S_{24}$ are
  removed and $S_{145}$ and $S_{245}$ are added in their place.  This remains a
  generating set because $S_{145} = S_{14}S_{15}S_{45}$ so that
  $S_{14}$ can be recovered from the other three, and likewise, $S_{245} = S_{24}S_{25}S_{45}$ so $S_{24}$ can be recovered.  The graph $C(J)$ is
  connected since both $S_{34}$ and $S_{45}$ commute with $S_{12}$.
  Since every standard generator (with the exception of $S_{14}$ and $S_{24}$)
  commutes with some element in $J$, and $S_{145}$ and $S_{245}$ commute with
  $S_{45}$, $J$ dominates $I$.  Lemma~\ref{lem:condom} completes the
  proof.
\end{proof}

At this point, the Star-or-Small Lemma, Lemma~\ref{lem:star-small},
 implies that our search for
characters in $\Sigma^1(\pure_n)^c$ can be restricted to those whose
graph is a star or a subgraph of $K_4$, plus possibly some isolated 
vertices.

\begin{lem}[Stars]\label{lem:star}
  If $\chi$ is a character of $\pure_n$ and the edges of $K_\chi$ form
  a star with at least $3$ edges, then \good.
\end{lem}

\begin{proof}
  If $\chi(\Delta)$ is not zero, then \good\ by
  Lemma~\ref{lem:zero-sum}.  Otherwise, permute the points so that
  $v_1$, $v_2$ and $v_3$ are leafs of $K_\chi$ and $v_4$ is the vertex
  all edges have in common.  Then let $I$ be the standard generators
  for this arrangement and let $J$ consist of the six elements
  $S_{14}$, $S_{24}$, $S_{34}$, $S_{A_1}$, $S_{A_2}$, and $S_{A_3}$
  where $A_i$ is the set $\{1,2, \ldots, n\}$ with $i$ removed.
  Because $\chi(\Delta)=0$ and $v_1$, $v_2$ and $v_3$ are leaves of
  $K_\chi$, we have that 
\[
 \chi(S_{A_i}) = \chi(\Delta)-\chi(S_{i4}) =
  -\chi(S_{i4}) \neq 0
\]
for $i\in \{1,2,3\}$.  In particular, all
  of $J$ survives under $\chi$.  Next since $\{j,4\} \subset A_i$ so
  long as $i$ and $j$ are distinct elements in $\{1,2,3\}$, we have
  that $S_{A_i}$ commutes with $S_{j4}$ in these situations.  As a
  consequence, the graph $C(J)$ is connected by a hexagon of edges.
  Finally, $J$ dominates $I$ since every standard generator avoids one
  of the first three points and thus commutes with one of the elements
  $S_{A_i}$.  Lemma~\ref{lem:condom} completes the proof.
\end{proof}

\begin{lem}[Disjoint leaves]\label{lem:disjoint-leaves}
  Let $\chi$ be a character of $\pure_n$.  If $K_\chi$ contains two
  vertices of valence~$1$ and the unique edges that end at these
  vertices are disjoint, then \good.
\end{lem}

\begin{proof}
  If $\chi(\Delta)$ is not zero, then \good\ by
  Lemma~\ref{lem:zero-sum}.  Otherwise, permute the points so that
  $v_1$ and $v_3$ are leaves and $e_{12}$ and $e_{34}$ are in
  $K_\chi$.  Then let $I$ be the standard generators for this
  arrangement and let $J$ consist of the five elements $S_{12}$,
  $S_{34}$, $S_{123}$, $S_{A_1}$ and $S_{A_3}$ where $A_i$ is the set
  $\{1,2, \ldots, n\}$ with $i$ removed.  Because $\chi(\Delta)=0$ and
  $v_1$ and $v_3$ are leaves of $K_\chi$, we have that $\chi(S_{A_1})
  = -\chi(S_{12})$, $\chi(S_{A_3}) = -\chi(S_{34})$, and
  $\chi(S_{123}) = \chi(S_{12})$.  In particular, all of $J$ survives
  under $\chi$.  The graph $C(J)$ is connected since both $S_{A_3}$
  and $S_{123}$ commutes with $S_{12}$ which commutes with $S_{34}$
  which commutes with $S_{A_1}$.  Finally, the set $J$ dominates $I$
  because the only standard generator which does not commute with
  $S_{A_1}$ or $S_{A_3}$ is $S_{13}$ and it commutes with $S_{123}$.
  Lemma~\ref{lem:condom} completes the proof.
\end{proof}

\begin{lem}[Disjoint edges and one triangle]\label{lem:triangle}
  Let $\chi$ be a character of $\pure_n$.  If $K_\chi$ contains a pair
  of disjoint edges and three of these endpoints form a triangle whose
  $\chi$-value is not zero, then \good.
\end{lem}

\begin{proof}
  Permute the points so that $e_{12}$ and $e_{34}$ are edges in
  $K_\chi$ and $\chi(S_{123})$ is not zero.  Then let $J = \{S_{12},
  S_{123}, S_{34}\}$ and let $I$ be the standard generating set for
  this arrangement with $S_{14}$ and $S_{24}$ removed and $S_{134}$
  and $S_{234}$ added in their place.  The set $I$ still generates
  $\pure_n$ since $S_{134} = S_{13} S_{14} S_{34}$ and $S_{234} =
  S_{23} S_{24} S_{34}$ so $S_{14}$ and $S_{24}$ can be recovered from
  the ones that remain.  The graph $C(J)$ is connected since $S_{12}$
  commutes with both $S_{123}$ and $S_{34}$.  Every standard generator
  with an endpoint outside the set $\{1,2,3,4\}$ commutes with either
  $S_{12}$ or $S_{34}$.  Of the six standard generators with both
  endpoints in this set, two are not in $I$, three commute with
  $S_{123}$, and $S_{34}$ commutes with itself.  Finally, the added
  elements $S_{134}$ and $S_{234}$ both commute with $S_{34}$ so $J$
  dominates $I$.  Lemma~\ref{lem:condom} completes the proof.
\end{proof}

\begin{figure}
  \begin{tikzpicture}[scale=1.5]
    \draw[-] (1,0)--(0,1);
    \draw[-] (0,0)--(1,0)--(1,1);
    \draw[-] (2,1)--(2,0);
    \draw[-] (3,0)--(3,1);
    \draw[-] (4,1)--(4,0)--(5,0)--(5,1);
    \foreach \x in {0,1,2,3,4,5}
    \foreach \y in {0,1} {\fill (\x,\y) circle (.6mm);}
  \end{tikzpicture}
  \caption{Three subgraphs of $K_4$.\label{fig:easy}}
\end{figure}
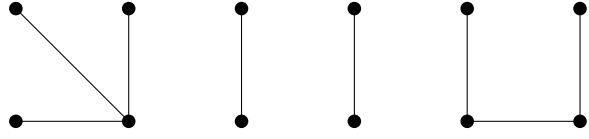

\begin{figure}
  \begin{tikzpicture}[scale=1.5]
    \draw[-] (1,0)--(0,1)--(0,0)--(1,0)--(1,1);
    \draw[-] (3,0)--(2,1)--(2,0)--(3,0)--(3,1)--(2,1);
    \draw[-] (5,1)--(4,1)--(4,0)--(5,0)--(5,1);
    \draw[-] (7,0)--(6,1)--(6,0)--(7,0)--(7,1)--(6,1);
    \draw[-] (6,0)--(7,1);
    \foreach \x in {0,1,2,3,4,5,6,7}
    \foreach \y in {0,1} {\fill (\x,\y) circle (.6mm);}
  \end{tikzpicture}
  \caption{Four subgraphs of $K_4$.\label{fig:medium}}
\end{figure}
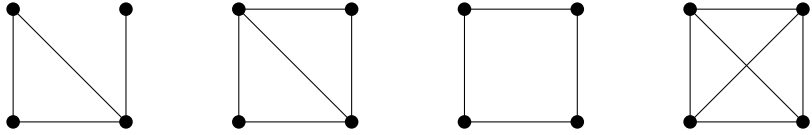

These lemmas combine to prove the following.

\begin{thm}[Characters in the invariant]\label{thm:invariant}
  Let $\chi$ be a character of $\pure_n$. If $[\chi]$ is not part of a
  $\pure_3$-circle or a $\pure_4$-circle then \good.
\end{thm}

\begin{proof}
  Let $\Gamma$ be the graph $K_\chi$ with isolated vertices removed.
  By Lemma~\ref{lem:disjoint-pair}, Lemma~\ref{lem:star}, and
  Lemma~\ref{lem:star-small}, $[\chi]$ is in $\Sigma^1(\pure_n)$
  unless $\Gamma$ has at most $4$ vertices. By
  Lemma~\ref{lem:zero-sum}, \good\ unless the sum of the edge weights
  is zero.  So assume that $\Gamma$ has at most $4$ vertices, none of
  them isolated and the sum of the edge weights is zero.  There are no
  $2$ vertex graphs satisfying these conditions and the only $3$
  vertex graphs remaining are those which represent characters in
  $\pure_3$-circles.  Thus we may also assume that $\Gamma$ has
  exactly $4$ vertices.  Up to isomorphism there are precisely seven
  such graphs and they are shown in Figures~\ref{fig:easy}
  and~\ref{fig:medium}.  If $\Gamma$ is isomorphism to one of three
  graphs in Figure~\ref{fig:easy}, then \good\ by Lemma~\ref{lem:star}
  or Lemma~\ref{lem:disjoint-leaves}.  Finally, the four graphs in
  Figure~\ref{fig:medium} all have disjoint edges.  If any triple of
  vertices have edges whose $\chi$-values have a non-zero sum, then
  \good\ by Lemma~\ref{lem:triangle}.  The only remaining case is
  where all such triples sum to zero.  By Lemma~\ref{lem:triple-sums}
  this means that disjoint edges are assigned equal values.  This
  rules out the two graphs on the left of Figure~\ref{fig:medium} and
  reduces the other two to graphs representing characters in
  $\pure_4$-circles.  And this completes the proof.
\end{proof}

Theorem~\ref{thm:complement} and Theorem~\ref{thm:invariant}
prove Theorem~\ref{thm:main}.  

\bibliography{refs-bns} 
\bibliographystyle{amsalpha}
\end{document}